\documentclass[a4paper,fleqn,reqno]{amsart}
\pdfoutput=1
\usepackage{times}

\usepackage[utf8]{inputenc}
\usepackage[
            T1]{fontenc}
\usepackage[
            ngerman,
            british]{babel}
\usepackage{amssymb}

\usepackage[left=26mm,right=26mm,top=44mm,bottom=33mm]{geometry}

\usepackage{xcolor}
\usepackage[
        colorlinks=false, pdfborder={0 0 .5}, 
        citecolor=DarkGreen,linktoc=all,unicode=true,
 	pdfauthor={Alexander Dyachenko},
 	pdfsubject={MSC(2010) 30E05, 46C20},
 	pdfkeywords={Nevanlinna-Pick interpolation, Indefinite metric, Pontryagin spaces},
 	pdfpagemode=UseOutlines,pdflang=en-GB,
        pdfa=true]{hyperref}
\hypersetup {
 	pdftitle={Functions of classes N⁺\_ϰ},
}

\usepackage[selected=true,shrink=10,stretch=30,letterspace=30]{microtype}
        
\usepackage{extarrows}
\usepackage{colonequals}

\theoremstyle{plain}
\newtheorem{theorem}{Theorem}
\newtheorem{prop*}[theorem]{Proposition}
\theoremstyle{definition}
\newtheorem{remark*}[theorem]{Remark}
\newtheorem{definition*}[theorem]{Definition}

\renewcommand{\le}{\leqslant}
\renewcommand{\ge}{\geqslant}

\begin{document}

\subjclass[2010]{30E05, 46C20}
\keywords{Nevanlinna-Pick interpolation, Indefinite metric, Pontryagin spaces, Moment problem}
\title{Functions of classes~$\mathcal N_\varkappa^+$}
\author{Alexander Dyachenko}
\email{dyachenk@math.tu-berlin.de}
\email{diachenko@sfedu.ru}
\thanks{This work was financially supported by the European Research Council under the
    European Union's Seventh Framework Programme (FP7/2007--2013)/ERC grant agreement no.
    259173.}
\address{TU-Berlin, MA 4-2, Stra\ss e des 17. Juni 136, 10623 Berlin, Germany}

\begin{abstract}
    In the present note we give an elementary proof of the necessary and sufficient condition
    for a univariate function to belong the class~$\mathcal N_\varkappa^+$. This class was
    introduced mainly to deal with the indefinite version of the Stieltjes moment problem (and
    corresponding $\pi$-Hermitian operators), although it is applicable beyond the original
    scope. The proof relies on asymptotic analysis of the corresponding Hermitian forms. Our
    result closes a gap in the criterion given by Krein and Langer in their joint paper of~1977.
    The correct condition was stated by Langer and Winkler in~1998, although they provided no
    proper reasoning.
\end{abstract}
\maketitle

\vspace{-1.5em}
\section{Introduction}
The function classes~$\mathcal N_\varkappa$ with $\varkappa=0,1,\dots$ were introduced in the
prominent paper~\cite{KreinLanger1} of M.~Krein and H.~Langer. They serve as a natural
generalisation of the Nevanlinna class~$\mathcal N \colonequals \mathcal N_0$ of all holomorphic
mappings~$\mathbb{C}_+\to\mathbb{C}_+$, which are also known as~$\mathcal R$-functions
(here~$\mathbb{C}_+\colonequals\{z\in\mathbb{C}:\Im z>0\}$ is the upper half of the complex
plane). A function~$\varphi(z)$ belongs to~$\mathcal N_\varkappa$ whenever it is meromorphic
in~$\mathbb{C}_+$, for any set of non-real points~$z_1$, $z_2$, \dots, $z_k$ the Hermitian
form
\begin{equation}\label{eq:quad_form}
    h_\varphi(\xi_1, \dots, \xi_k \vert z_1,\dots, z_k)\colonequals
    \sum_{n,m=0}^k \frac{\varphi(z_m) - \overline{\varphi(z_n)}}
            {z_m-\overline z_n}\xi_m\overline\xi_n
\end{equation}
has at most~$\varkappa$ negative squares and for some set of points there are
exactly~$\varkappa$ negative squares. It is convenient (and generally accepted) to define
$\mathcal N_\varkappa$\nobreakdash-functions in the lower half of the complex plane by complex
conjugation, i.e.~$\overline{\varphi(z)}=\varphi(\overline z)$.
A significant particular case is presented by the classes~$\mathcal N_\varkappa^+$, which are
considered here. They contain all $\mathcal N_\varkappa$-functions~$\varphi(z)$ such
that~$z\varphi(z)$ belongs to~$\mathcal N$. Among various applications, $\mathcal N$, \
$\mathcal N_\varkappa$ and~$\mathcal N_\varkappa^+$ appear in the moment problems and have
connections to the spectral theory of operators. However, the classes~$\mathcal N_\varkappa^+$
can find even more applications as a foremost generalisation of the Stieltjes
functions~$\mathcal N_0^+$.

This note aims at obtaining the necessary and sufficient condition for a function to be in the
class~$\mathcal N_\varkappa^+$ through the asymptotic analysis of the corresponding Hermitian
forms. As a main tool, we use the basic Nevanlinna-Pick theory for the halfplane within the
framework presented, for example, in~\cite[Chapter~3]{Akhiezer65}
or~\cite[Chapters~II--III]{Donoghue}. We show that, roughly speaking, $\mathcal N_\varkappa^+$
differs from the Stieltjes class~$\mathcal N_0^+$ in having~$\varkappa$ simple negative poles,
one of which can reach the origin and merge there into another singularity. More precisely,
\emph{a function~$\varphi(z)$ belongs to the class~$\mathcal N_\varkappa^+$ if and only if it
    has one of the forms
\begin{gather}
    \tag{A}\label{eq:repr_th_3.8_a}
    \varphi(z)=s_0+\sum_{j=1}^\varkappa\frac{\gamma_j}{\alpha_j-z}+\int_0^\infty\frac{d\nu(t)}{t-z};\\
    \tag{B}\label{eq:repr_th_3.8_b}
    \varphi(z)=s_0+\frac{s_1}{z}-\frac{s_2}{z^2}+\sum_{j=1}^{\varkappa-1}\frac{\gamma_j}{\alpha_j-z}
    +\int_0^\infty\frac{d\nu(t)}{t-z},
    \quad\text{where }
    \max\{s_1,s_2\}>0,\ \nu(+0)=0;\\
    \notag
    \varphi(z)=s_0+\frac{s_1}{z}-\frac{s_2}{z^2}+\sum_{j=1}^{\varkappa-1}\frac{\gamma_j}{\alpha_j-z}
    +\frac 1z\int_0^\infty\left(\frac{1}{t-z}-\frac t{1+t^2}\right)d\sigma(t),\\
    \tag{C}\label{eq:nu_inf}
    \hspace{52mm}
    \text{where }
    \int_0^1\frac{d\sigma(t)}t = \infty,\ \sigma(0+)=0,\ 
    \int_0^{\infty}\frac{d\sigma(t)}{1+t^2} < \infty.
\end{gather}
Here $s_0,s_2\ge 0$, ~$s_1\in\mathbb{R}$, ~$\gamma_j,\alpha_j<0$ for~$j=1,2,\dots,\varkappa$ and
$\nu(t),\sigma(t)$ are nondecreasing left-continuous functions such that~$\nu(0)=\sigma(0)=0$
and~$\int_0^{\infty}\frac{d\nu(t)}{1+t} < \infty$.}
The function~$\sigma(t)$ is intentionally denoted by a distinct letter to emphasize that its
constraints at infinity are weaker. For our purposes, it is more convenient to formulate this
criterion as Theorem~\ref{th:1}: to give the corresponding formulae for functions~$\Phi(z)$ such
that~$\varphi(z)\colonequals\frac 1z\Phi(z)$ belong to the class~$\mathcal N_\varkappa^+$. At
that, the case~\eqref{eq:repr_th_3.8_a} corresponds to the
representation~\eqref{eq:phi_form_2_with_cond}, and the cases~\eqref{eq:repr_th_3.8_b}
and~\eqref{eq:nu_inf} correspond to the representation~\eqref{eq:phi_form_1_with_cond}. The
intermediary Propositions~\ref{prop:1} and~\ref{prop:2} are more general than required for
proving Theorem~\ref{th:1} and can be interesting per se.

This result corrects Theorem~3.8 of~\cite{KreinLanger1}: the authors put the
condition~$\int_0^{\infty}\frac{d\sigma(t)}{1+t} < +\infty$ in the case~\eqref{eq:nu_inf}. As a
result, 
Theorem~3.8 fails to address~$\mathcal N_\varkappa^+$\nobreakdash-functions like
\[
\psi(z)=\frac 1z\cdot\left(\frac{1}{\sqrt{z}}\cot\frac{1}{\sqrt{z}}-\sqrt{z}\cot\sqrt{z}\right)
\]
with~$\varkappa=1$. More than likely, this mistake is just an oversight: for proving the
representations~\eqref{eq:repr_th_3.8_a}--\eqref{eq:repr_th_3.8_b} the authors use, in fact, the
measure~$\frac{d\sigma(t)}t$ as~$d\nu(t)$; then they put~$\nu(t)$ instead of~$\sigma(t)$
in~\eqref{eq:nu_inf}. Furthermore, their proof seems to be less transparent since it involves
operator theory (which is convenient for the general case of~$\mathcal N_\varkappa$). Lemma~5.3
from~\cite[p.~421]{LangerWinkler} (see Theorem~\ref{th:1} herein) has a proper statement, and
the function~$\psi(z)$ is allowed as an entry of~$\mathcal N_1^+$. Unfortunately, the proof
in~\textup{\cite{LangerWinkler}} is invalid. Our proof does not depend on the results
of~\cite{KreinLanger1,LangerWinkler}.

It is worth noting that the gap in~\cite[Theorem~3.8]{KreinLanger1} does not affect dependent
results which assume that the function~$z\varphi(z)$ has an asymptotic expansion of the
form~$\sum_{n=0}^\infty c_nz^n$ as $z\to i\cdot 0+$ or of the form~$\sum_{n=0}^\infty c_nz^{-n}$
as $z\to+\infty\cdot i$. Indeed, in the former case~$\varphi(z)$ can be expressed as
in~\eqref{eq:repr_th_3.8_a} or~\eqref{eq:repr_th_3.8_b}, and in the latter case our
representation~\eqref{eq:nu_inf} reduces (see~\cite[p.~143]{Kac}) to the item~3.\ of Theorem~3.8
in~\cite{KreinLanger1}. On the other hand, the function~$z\psi(z)$, where~$\psi(z)$ is given
above, has no such asymptotic expansions.

\section{Preliminaries}\label{sec:n-functions}
Each $\mathcal N$-function~$\Phi$ has the following integral representation
(see \emph{e.g.}~\cite[p.~92]{Akhiezer65}, \cite[p.~20]{Donoghue}):
\begin{equation}\label{eq:Phi_N}
    \Phi(z)= bz + a + \int_{-\infty}^\infty\left(\frac{1}{t-z}-\frac{t}{1+t^2}\right)d\sigma(t)
\end{equation}
where $a$ is real, $b\ge 0$ and $\sigma(t)$ is a real non-decreasing function satisfying
$\int_{-\infty}^\infty \frac{d\sigma(t)}{1+t^2} < \infty$. The converse is also true: all functions of
the form~\eqref{eq:Phi_N} belong to~$\mathcal N$.

To be definite, we assume that the function~$\sigma(t)$ is \emph{left-continuous}, that
is~$\sigma(t)=\sigma(t-)$ for all~$t\in\mathbb{R}$. Accordingly, the notation for integrals
with respect to~$d\sigma(t)$ is as in the formula
\( \int_\alpha^\beta f(t)\,d\sigma(t) \colonequals \int_{[\alpha,\beta)} f(t)\,d\sigma(t) \) for
arbitrarily taken real numbers~$\alpha,\beta$ and function~$f(t)$.

\begin{remark*}\label{rem:defN} 
    A function~$\Phi$ given by the formula~\eqref{eq:Phi_N} is holomorphic outside the real
    line. Furthermore, it has an analytic continuation through the intervals outside the support
    of $d\sigma$. The function $\varphi(z)\colonequals\Phi(z)/z$ has the same singularities with
    the exception of the origin (generally speaking). We can additionally note that,
    \[
    \text{if}\quad z_1<z_2<t \quad\text{or}\quad t<z_1<z_2,\quad\text{then}\quad
    \frac 1{t-z_2} - \frac 1{t-z_1} =\frac {z_2-z_1}{(t-z_1)(t-z_2)}>0.
    \]
    Consequently, given a real interval~$(\alpha,\beta)$ that has no common points with the
    support of~$d\sigma$, the condition $\alpha<z_1<z_2<\beta$ implies $\Phi(z_1)<\Phi(z_2)$
    unless~$\Phi(z)\equiv a$, which is seen from the representation~\eqref{eq:Phi_N}. (This fact
    is also seen immediately from the definition of the class~$\mathcal N$:
    see~\cite[p.~18]{Donoghue}.) Put in other words, the function~$\Phi(z)$ increases in the
    interval~$(\alpha,\beta)$ unless it is identically constant.
\end{remark*}

\begin{theorem}[Coincides with Lemma~5.3 from~{\cite[p.~421]{LangerWinkler}}]\label{th:1}
    Let a function~$\Phi\in\mathcal N$. The function $\varphi(z)\colonequals \Phi(z)/z$ belongs
    to~$\mathcal N_\varkappa$ if and only if the representation~\eqref{eq:Phi_N} of~$\Phi(z)$ is
    either of the form
    \begin{subequations} \label{eq:phi_form_1_with_cond}
    \begin{equation} \label{eq:phi_form_1}
        \Phi(z)= bz + a + \sum_{n=1}^{\varkappa-1}\frac{\sigma_n}{\lambda_n-z}
        + \int_0^\infty\left(\frac{1}{t-z}-\frac{t}{1+t^2}\right)d\sigma(t)
    \end{equation}
    where~$\lambda_n<0$, \ $n=1,2,\dots,\varkappa-1$ and
    \begin{equation} \label{eq:phi_form_1_cond}
        0<\Phi(0-)
        = a + \sum_{n=1}^{\varkappa-1}\frac{\sigma_n}{\lambda_n}
        + \int_0^\infty\frac{d\sigma(t)}{t+t^3}\le\infty,
    \end{equation}
    \end{subequations}
    or of the form
    \begin{subequations}\label{eq:phi_form_2_with_cond}
    \begin{equation} \label{eq:phi_form_2}
        \Phi(z)= bz + a + \sum_{n=1}^{\varkappa}\frac{\sigma_n}{\lambda_n-z}
        + \int_0^\infty\left(\frac{1}{t-z}-\frac{t}{1+t^2}\right)d\sigma(t),
    \end{equation}
    where~$\lambda_n<0$, \ $n=1,2,\dots,\varkappa$ and
    \begin{equation}\label{eq:phi_form_2_cond}
        \Phi(0-) = a + \sum_{n=1}^{\varkappa}\frac{\sigma_n}{\lambda_n} +
        \int_0^\infty\frac{d\sigma(t)}{t+t^3}\le 0.
    \end{equation}
    \end{subequations}
\end{theorem}
Note that~$\mathcal N$-functions of the forms~\eqref{eq:phi_form_1} and~\eqref{eq:phi_form_2}
have the corresponding limit~$\Phi(0-)$ defined, because (when non-constant) they grow
monotonically (see Remark~\ref{rem:defN}) outside the support of corresponding
measure~$d\sigma(t)$. Moreover, all numbers~$\sigma_n$
in~\eqref{eq:phi_form_1_with_cond}--\eqref{eq:phi_form_2_with_cond} are positive due to the
condition~$\Phi\in\mathcal N$.

\section{Proofs}\label{sec:proofs}
\begin{definition*}\label{def:poi_inc}
    A real point~$\lambda$ is called a \emph{point of increase} of a function~$\sigma(t)$ if
    $\sigma(\lambda+\varepsilon)>\sigma(\lambda-\varepsilon)$ for every~$\varepsilon>0$ small
    enough. In particular, the set of all points of increase of a non-decreasing
    function~$\sigma(t)$ is the support of~$d\sigma(t)$.
\end{definition*}
In each punctured neighbourhood of the point of increase~$\lambda$ there exists~$\lambda'$ such
that the limit
\[
\sigma'(\lambda')
= \lim_{\varepsilon\to 0}
    \frac{\sigma(\lambda'+\varepsilon)-\sigma(\lambda'-\varepsilon)}{2\varepsilon}
\]
is positive or nonexistent. Indeed, otherwise $\sigma'(t)\le 0$ in the closed
interval~$-\varepsilon\le t\le \varepsilon$ for some~$\varepsilon$ small enough, thus
integrating~$\sigma'(t)$ over this interval leads us to a contradiction. Consequently, if we
know that the function~$\sigma(t)$ has at least $\varkappa$ negative points of increase, then we
always can select $\varkappa$ points of
increase~$\lambda_\varkappa<\lambda_{\varkappa-1}<\dots<\lambda_1<0$, in which the
derivative~$\sigma'$ is nonexistent or positive. Given such a set of points, denote
\begin{equation}\label{eq:def_delta}
\delta\colonequals
\frac 13 \min\Big\{-\lambda_1,\min_{1\le n\le\varkappa-1}(\lambda_{n}-\lambda_{n+1})\Big\},
\end{equation}
put $U_n\colonequals(\lambda_n-\delta,\lambda_n+\delta)$, where $n=1,\dots,\varkappa$, and
$U_0\colonequals\mathbb{R}\setminus\left(\bigcup_{n=1}^{\varkappa} U_n\right)$.

\begin{prop*}\label{prop:1}
    Consider a function~$\Phi(z)=z\varphi(z)$ of the form~\eqref{eq:Phi_N}. Let the
    function~$\sigma(t)$ have at least~$\varkappa$ negative points of
    increase~$\lambda_\varkappa<\dots<\lambda_1$ in which the derivative~$\sigma'$ is
    nonexistent or positive. Then the Hermitian
    form~\( h_\varphi(\xi_1, \dots, \xi_\varkappa \vert \lambda_1+i\eta,\dots,
    \lambda_\varkappa+i\eta) \) defined in~\eqref{eq:quad_form} has~$\varkappa$ negative squares
    for some small values of~$\eta>0$.
\end{prop*}
\begin{proof}
    Since
    \begin{subequations}\label{eq:ident_1tz_1t_2}
    \begin{align}
      \label{eq:ident_1tz}
      \frac{1}{t-z} &= \frac{z+t-z}{t(t-z)} = \frac{z}{t(t-z)}+\frac{1}{t}
                      \quad \text{and}\\
      \label{eq:ident_1t_2}
      \frac{1}{t}-\frac{t}{1+t^2} &= \frac{1+t^2 -t^2}{t(1+t^2)} = \frac{1}{t(1+t^2)},
    \end{align}
    \end{subequations}
    from the expression~\eqref{eq:Phi_N} we obtain
    \begin{equation}\label{eq:ident_phi_int}
        \begin{aligned}
        \varphi(z)={} & b + \frac az
           + \frac 1z \int_{-\infty}^\infty\left(\frac{1}{t-z}-\frac{t}{1+t^2}\right)d\sigma(t)
        \\
            ={}& b + \frac az
              + \int_{-\infty}^\infty\left(\frac{1}{t-z}+\frac{1}{z(1+t^2)}\right)\frac{d\sigma(t)}t
            = \sum_{n=0}^{\varkappa} \varphi_n(z),
        \end{aligned}
    \end{equation}
    where
    \[
        \varphi_n(z)
        = \int_{U_n}\frac{1}{t-z}\cdot\frac{d\sigma(t)}{t},\quad n=1,\dots,\varkappa,
    \]
    are the terms dominant on the intervals~$U_n$, and
    \[
    \varphi_0(z)
        \colonequals b + \frac {\widetilde a}z
           + \frac 1z \int_{U_0}\left(\frac{1}{t-z}-\frac{t}{1+t^2}\right)d\sigma(t)
           \quad\text{with}\quad
    \widetilde a\colonequals a + \sum_{n=1}^{\varkappa} \int_{U_n}\frac{d\sigma(t)}{t+t^3}
    \]
    contains the remainder term. Let us additionally assume $z_n\colonequals\lambda_n+i\eta$.

    On the one hand, each function~$\varphi_n$ is holomorphic outside~$U_n$; therefore,
    when~$m,k\ne n$ we have that the limits
    \begin{equation}\label{eq:lim_boundary}
        \lim_{\eta\to 0}\frac{\varphi_n(z_m) - \varphi_n(\overline z_k)}{z_m-\overline z_k}
        =
        \begin{cases}
            \dfrac{\varphi_n(\lambda_m) - \varphi_n(\lambda_k)}{\lambda_m-\lambda_k},
            &\text{if } k\ne m,\\
            \varphi'_n(\lambda_m),
            &\text{if } k=m\\
        \end{cases}
    \end{equation}
    are finite as~$n=0,\dots,\varkappa$. On the other hand, with the notation
    $\widetilde\sigma(t)=
    \int_{-\lambda_\varkappa-\delta}^t|s|^{-1}d\sigma(s)$, that is
    $\widetilde\sigma(t)=
    -\int_{-\lambda_\varkappa-\delta}^ts^{-1}d\sigma(s)$ when~$t<0$, for~$n\ne 0$ we have
    \[
        \begin{aligned}
            \rho^2_n(\eta)\colonequals{}&
            -\frac{\varphi_n(z_n) - \varphi_n(\overline z_n)}{z_n-\overline z_n}
            ={}\int_{U_n}\frac{t - \overline z_n - t + z_n}
                       {(z_n-\overline z_n)(t-z_n)(t-\overline z_n)}d\widetilde\sigma(t)
            ={}\int_{U_n}\frac{d\widetilde\sigma(t)}{|t-z_n|^2}\\
            \ge{}&\int_{\lambda_n-\eta}^{\lambda_n+\eta}\frac{d\widetilde\sigma(t)}{(t-\lambda_n)^2+\eta^2}
            \ge\int_{\lambda_n-\eta}^{\lambda_n+\eta}\frac{d\widetilde\sigma(t)}{2\eta^2}
            =
            \frac 1\eta \cdot
            \frac {\widetilde\sigma (\lambda_n+\eta)-\widetilde\sigma (\lambda_n-\eta)}{2\eta}.
        \end{aligned}
    \]
    Consequently, $\limsup_{\eta\to0+}\left(\eta\cdot\rho^2_n(\eta)\right)$ is positive
    or~$+\infty$ because~$\lambda_n$ is a point of increase of~$\sigma(t)$. Moreover, we
    fix~$\rho_n(\eta)>0$ for definiteness. In terms of big $O$ notation, there exists a
    sequence of positive numbers~$\eta_1, \eta_2, \dots$ tending to zero such that
    \begin{equation}\label{eq:lim_rho}
        \frac 1{\rho_n(\eta_k)} = O\left(\sqrt{\eta_k}\right)
        \quad\text{as}\quad k\to +\infty
        \quad\text{and}\quad n=1,\dots,\varkappa.
    \end{equation}

    According to~\eqref{eq:lim_boundary}, we additionally have
    \begin{equation} \label{eq:phi_n_n}
        -\frac{\varphi(z_n) - \varphi(\overline z_n)}{z_n-\overline z_n}=
        -\frac{\varphi_n(z_n) - \varphi_n(\overline z_n)}{z_n-\overline z_n}
        -\sum_{m \ne n}\frac{\varphi_m(z_n) - \varphi_m(\overline z_n)}{z_n-\overline z_n}
        =\rho^2_n(\eta)+O(1)
    \end{equation}
    when~$\eta$ is assumed to be small. Furthermore,
    \[
        -\frac{\varphi_n(z_n) - \varphi_n(\overline z_m)}{z_n-\overline z_m}
        =\int_{U_n}\frac{t - \overline z_m - t + z_n}
                       {(z_n-\overline z_m)(t-z_n)(t-\overline z_m)}d\widetilde\sigma(t)
        =\int_{U_n} \frac{d\widetilde\sigma(t)}{(t-z_n)(t-\overline z_m)},
    \]
    which implies (with the help of the elementary inequality
    $2\alpha\beta\le \frac{\alpha^2}{c} + c\beta^2$ valid for any positive numbers)
    \begin{equation}\label{eq:bound_mixed}
    \begin{aligned}
            \left|\frac{\varphi_n(z_n) - \varphi_n(\overline z_m)}{z_n-\overline z_m}\right|
            &\le \int_{U_n} \frac{d\widetilde\sigma(t)}{|t-z_n|\cdot|t-\overline z_m|}
            \le \int_{U_n} \frac{d\widetilde\sigma(t)}{2\rho_n(\eta)|t-z_n|^2}
          + \int_{U_n} \frac{\rho_n(\eta)\,d\widetilde\sigma(t)}{2|t-z_m|^2}
          \\
        &= \frac {1}{2\rho_n(\eta)}\rho^2_n(\eta)
          + \frac {\rho_n(\eta)}{2} \int_{U_n} \frac{d\widetilde\sigma(t)}{|t-z_m|^2}
        \le C(n,m,\delta)\rho_n(\eta),
    \end{aligned}
    \end{equation}
    because the distance between~$U_n$ and~$z_m$ is more than~$\delta$. The
    factor~$C(n,m,\delta)>0$ in~\eqref{eq:bound_mixed} is independent of~$\eta$. The finiteness
    of~\eqref{eq:lim_boundary} gives
    \[
    \frac{\varphi(z_n) - \varphi(\overline z_m)}
    {z_n-\overline z_m}
    =
    \frac{\varphi_n(z_n) - \varphi_n(\overline z_m)}
    {z_n-\overline z_m}
    + \frac{\varphi_m(z_n) - \varphi_m(\overline z_m)}
    {z_n-\overline z_m}
    + O(1)\quad\text{as}\quad\eta\to 0+.
    \]
    This can be combined with the estimate~\eqref{eq:bound_mixed}, thus giving us for
    small~$\eta$
    \begin{equation} \label{eq:phi_n_m_est}
                \left|\frac{\varphi(z_n) - \varphi(\overline z_m)}
                    {(z_n-\overline z_m)\rho_n(\eta)\rho_m(\eta)}\right|
            \le
            \frac{C(n,m,\delta)}{\rho_m(\eta)}+\frac{C(m,n,\delta)}{\rho_n(\eta)}
            +O\left(\frac{1}{\rho_n(\eta)\rho_m(\eta)}\right).
    \end{equation}
    The relations~\eqref{eq:phi_n_n} and~\eqref{eq:phi_n_m_est} allow us to make the final step
    in the proof. The substitution $\xi_n \mapsto \zeta_n/\rho_n(\eta)$ gives us
    \begin{gather}\label{eq:h_is_neg_def}
        h_\varphi\left(\frac{\zeta_1}{\rho_1(\eta)},\dots,\frac{\zeta_\varkappa}{\rho_\varkappa(\eta)}
        \Big\vert z_1,\dots,z_\varkappa\right)
            = \!\!\sum_{n,m=1}^{\varkappa} \!\!
            \frac{\varphi(z_n) - \varphi(\overline z_m)} {z_n-\overline z_m}
            \cdot \frac{\zeta_n\overline\zeta_m}{\rho_n(\eta)\rho_m(\eta)}
            = R(\eta) -\sum_{n=1}^{\varkappa} |\zeta_n|^2,
            \\
            \label{eq:h_is_neg_def_R}
        \text{where}\quad
        \big|R(\eta)\big| = \sum_{n=1}^{\varkappa} \Bigg(
            |\zeta_n|^2O\bigg(\frac 1{\rho_n^2(\eta)}\bigg)
            + \sum_{m\ne n}
            \zeta_n\overline\zeta_m
            O\bigg(\frac 1{\rho_m(\eta)} + \frac 1{\rho_n(\eta)}\bigg)
        \Bigg)
        .
    \end{gather}
    According to~\eqref{eq:lim_rho}, in each neighbourhood of zero we can
    choose~$\eta\in\big\{\eta_k\big\}_{k=1}^\infty$, such that the inequality
    \( \frac 1{\rho_n(\eta)} \le M_n\sqrt{\eta} \) holds true for a fixed number~$M_n>0$
    dependent only on~$\varphi$ and~$n$. For such choice of~$\eta$, the
    estimate~\eqref{eq:h_is_neg_def_R} implies
    $\big|R(\eta)\big|\le M\sqrt{\eta}\cdot\sum_{n=1}^{\varkappa} |\zeta_n|^2$ with some
    fixed~$M$. Therefore, the sign of the Hermitian form~\eqref{eq:h_is_neg_def} will be
    determined by the last term~$-\sum_{n=1}^{\varkappa} |\zeta_n|^2$ alone for every set of
    complex numbers~$\{\zeta_1,\dots,\zeta_{\varkappa}\}$ as soon
    as~$\eta\in\big\{\eta_k\big\}_{k=1}^\infty$ is small enough.
\end{proof}

\begin{prop*}\label{prop:2}
    Under the conditions of Proposition~\ref{prop:1} assume that $\Phi(z)$ is regular in the
    interval~$(-\varepsilon,0)$ and $0<\Phi(0-)\le\infty$. Then the Hermitian form
    \( h_\varphi(\xi_0, \dots, \xi_\varkappa \vert z_0, \dots, z_\varkappa) \),
    where~$z_m=\lambda_m+i\eta$ for $m=1,\dots,\varkappa$ and $z_0=-\sqrt\eta+i\mu$, is negative
    definite when the numbers~$\eta>0$ and~$\mu>0$ are chosen appropriately.
\end{prop*}

\begin{proof}
    Split the function~$\varphi(z)$ into two parts $\psi_0(z)$ and~$\psi_1(z)$ such that
    $\varphi(z)=\psi_0(z)+\psi_1(z)$ and
    \[
    \begin{aligned}
        \psi_1(z) &\colonequals
        b+\frac 1z
        \int_{\mathbb R\setminus(-\varepsilon,\varepsilon)}\left(\frac{1}{t-z}-\frac{t}{1+t^2}
            -\frac{1}{t+t^3}\right)d\sigma(t).
    \end{aligned}
    \]
    The integral here is analytic for~$|z|<\varepsilon$ and vanishes at the origin
    (see~\eqref{eq:ident_1t_2}). The function~$\psi_1(z)$ therefore is also analytic
    for~$|z|<\varepsilon$. The part~$\psi_0(z)$ has the form
    \begin{align*}
        \psi_0(z) \colonequals{}& \frac az
        + \frac 1z \int_{(-\varepsilon,\varepsilon)}\left(\frac{1}{t-z} - \frac{t}{1+t^2}\right)d\sigma(t)
        + \frac 1z \int_{\mathbb R\setminus(-\varepsilon,\varepsilon)}\frac{d\sigma(t)}{t+t^3}\\
        ={}& \frac Az
        + \frac 1z \int_0^\varepsilon\frac{d\sigma(t)}{t-z},
        \quad\text{where we put}\quad A\colonequals
        a + \int_{\mathbb R\setminus(-\varepsilon,\varepsilon)}\frac{d\sigma(t)}{t+t^3}
        - \int_0^\varepsilon\frac{td\sigma(t)}{1+t^2},
    \end{align*}
    i.e. $A$ is a finite real constant. The integral over~$(-\varepsilon,0)$ is zero in the
    representation of~$\psi_0$, because the function~$\Phi(z)$ is regular in this interval, and
    thus~$\sigma(t)$ is constant for $-\varepsilon<t<0$.

    First assume that $x$ varies on $(-\varepsilon,0)$ close enough to~$0$, so that $\Phi(x)>3M$
    with some fixed~$M>0$. On the one hand, one of the Cauchy-Riemann equations and the
    condition $\Phi'(x)\ge 0$ (see Remark~\ref{rem:defN}) imply
    \begin{equation*}
        \frac{\partial\Im\varphi(x+iy)}{\partial y}\bigg|_{y=0}
        = \frac{d}{dx}\varphi(x)
        = \frac{d}{dx}\frac{\Phi(x)}{x}
        = \frac{\Phi'(x)x - \Phi(x)}{x^2}
        < -\frac{3M}{x^2},
    \end{equation*}
    Given~$x$ we can chose~$\mu\in(0,-x)$ such that
    \begin{equation}\label{eq:choice_of_mu}
        \left|\frac{\partial\Im\varphi(x+iy)}{\partial y}\bigg|_{y=0} -
            \frac{\varphi(x+i\mu)-\varphi(x-i\mu)}{2i\mu}\right|\le
        \frac{M}{x^2}
    \end{equation}
    relying on the fact that~$\varphi(x+iy)$ is smooth for real~$y$. The last two inequalities
    together imply that
    \[
        \frac{\varphi(z_0)-\varphi(\overline z_0)}{z_0-\overline z_0}
        =\frac{\varphi(x+i\mu)-\varphi(x-i\mu)}{(x+i\mu)-(x-i\mu)}
        < -\frac{3M}{x^2} + \frac{M}{x^2}
        = -\frac{2M}{x^2}
    \]
    for $z_0=x+i\mu$. Therefore,
    \begin{equation}\label{eq:phi_prime_at_zero}
        \rho_0^2(x^2)\colonequals
        -\frac{\psi_0(z_0)-\psi_0(\overline z_0)}{z_0-\overline z_0}
        =-\frac{\varphi(z_0)-\varphi(\overline z_0)}{z_0-\overline z_0}
         +\frac{\psi_1(z_0)-\psi_1(\overline z_0)}{z_0-\overline z_0}
        \ge \frac{M}{x^2}
    \end{equation}
    for small enough~$|x|$ on account of the smoothness of~$\psi_1(z)$. We
    assume~$\rho_0(x^2)>0$ for definiteness.

    On the other hand, the definition of~$\psi_0(z)$ implies that
    \begin{equation}\label{eq:rho0_in_terms_psi0}
    \begin{aligned}
        \frac{\psi_0(z_0) - \psi_0(\overline z_m)}{z_0 - \overline z_m}
        ={}& \frac{A/z_0-A/\overline z_m}{z_0 - \overline z_m}
        + \frac 1{z_0\overline z_m(z_0-\overline z_m)}
        \int_0^\varepsilon\left(
            \frac{\overline z_m}{t-z_0}-\frac{z_0}{t-\overline z_m}
        \right)d\sigma(t)\\
        ={}&
        \frac 1{z_0\overline z_m}\left(-A
            + \int_0^\varepsilon
            \frac{\overline z_m(t-\overline z_m)- z_0(t-z_0)}
            {(t-\overline z_m)(t-z_0)(z_0-\overline z_m)}d\sigma(t)
        \right)\\
        ={}&
        -\frac 1{z_0\overline z_m}\left(A
            + \int_0^\varepsilon \frac{t - z_0 -\overline z_m}
            {(t-\overline z_m)(t-z_0)}d\sigma(t)\right),
        \quad
        m=0,\dots,\varkappa.
    \end{aligned}
    \end{equation}
    The inequalities~$\Big|\frac{-z_0}{t-z_0}\Big|\le 1$ and,
    hence,~$\Big|\frac{-z_0}{(t-\overline z_m)(t-z_0)}\Big|\le \frac{1}{|z_m|}$ are valid for
    all~$t\ge 0$. Let us apply the latter to estimating the absolute value of the expression
    \eqref{eq:rho0_in_terms_psi0}:
    \begin{equation}\label{eq:finite_difference_psi0}
    \begin{aligned}
            \left|\frac{\psi_0(z_0) - \psi_0(\overline z_m)}{z_0 - \overline z_m}\right|
        \le{}&
        \frac 1{|z_0z_m|}\left(|A|
            + \int_0^\varepsilon
            \left|\frac{- z_0}{(t-\overline z_m)(t-z_0)}
                + \frac{1}{(t-z_0)}
            \right|d\sigma(t)
        \right)\\
        \le{}&
        \frac {|A|}{|z_0z_m|}
        + \frac 1{|z_0z_m^2|} \int_0^\varepsilon d\sigma(t)
        + \frac 1{|z_0z_m|}\int_0^\varepsilon\frac {d\sigma(t)}{|t-z_0|}.
    \end{aligned}
    \end{equation}
    In particular, putting~$m=0$ in~\eqref{eq:finite_difference_psi0} gives us    
    \begin{equation}\label{eq:est_rho_0}
        \rho_0^2(x^2)=
        \left|\frac{\psi_0(z_0) - \psi_0(\overline z_0)}{z_0 - \overline z_0}\right|
        \le
        \frac 2{|z_0^3|}\int_0^\varepsilon d\sigma(t) + \frac{|A|}{|z_0^2|}
        \le
        2\frac {\sigma(\varepsilon-)-\sigma(0)}{|x|^3} + \frac{|A|}{x^2},
    \end{equation}
    which complements
    \begin{equation}\label{eq:est_psi_0_0}
        \left|\frac{\varphi(z_0) - \varphi(\overline z_0)}{z_0 - \overline z_0}
            +
                    \rho_0^2(x^2)
        \right|
        =
        \left|
            \frac{\psi_1(z_0) - \psi_1(\overline z_0)}{z_0 - \overline z_0}
        \right|
        = O\left(1\right)
        \quad\text{as}\quad x\to 0,
    \end{equation}
    where~$O\left(1\right)$ on the right-hand side does not depend on~$\mu\in(0,\varepsilon)$. Now recall
    that~$\Re z_\varkappa=\lambda_\varkappa<\dots<\Re z_1=\lambda_1<-\varepsilon$. Since
    $|t-z_0| = t-x +\mu < t-2x$ provided that~$t\ge 0$ and~$\mu<|x|$,
    from~\eqref{eq:finite_difference_psi0} we obtain
    \begin{equation}\label{eq:est_psi_0_n}
    \begin{aligned}
            \left|\frac{\psi_0(z_0) - \psi_0(\overline z_m)}{z_0 - \overline z_m}\right|
        &\le
        \frac 1{|z_0z_m^2|} \int_0^\varepsilon d\sigma(t)
        + \frac {|A|}{|z_0z_m|}
        + \frac 1{|z_0z_m|}\int_0^\varepsilon\frac {|t-z_0|}{|t-z_0|^2}d\sigma(t)
        \\
        &\le
        \frac 1{|z_0z_m^2|} \int_0^\varepsilon d\sigma(t)
        + \frac {2|A|}{|z_0z_m|}
        + \frac{|z_0|}{|z_m|}\left(
            \frac {A}{|z_0|^2}
            + \frac 1{|z_0|^2}\int_0^\varepsilon\frac {t-2x}{|t-z_0|^2}d\sigma(t)\right)\\
        &\xlongequal{\eqref{eq:rho0_in_terms_psi0}}
        \frac 1{|z_0z_m^2|} \int_0^\varepsilon d\sigma(t)
        + \frac {2|A|}{|z_0z_m|}
        + \frac{|z_0|}{|z_m|}\left(
            -\frac{\psi_0(z_0) - \psi_0(\overline z_0)}{z_0 - \overline z_0}\right)\\
        &\le
        \frac {\sigma(\varepsilon-)-\sigma(0)}{|x|\lambda_m^2}
        + \frac {2|A|}{|x\lambda_m|}
        + \frac{2|x|}{|\lambda_m|}\rho_0^2(x^2)
        =O\left(\frac 1x\right) + O\left(x\rho_0^2(x^2)\right)
        ,
    \end{aligned}
    \end{equation}
    where~$x$ tends to zero and~$m=1,\dots,\varkappa$.
    
    To implement the same technique as in the proof of Proposition~\ref{prop:1} it is enough to
    put~$x\colonequals-\sqrt\eta$ and to study the order of summands in the
    form~\(h_\varphi(\xi_0, \dots, \xi_\varkappa \vert z_0, \dots, z_\varkappa)\). In addition
    to~$x>-\varepsilon$, we suppose that~$x>-\delta$, where the positive number~$\delta$ is
    defined in~\eqref{eq:def_delta}; consequently~$\eta<\min\{\varepsilon^2,\delta^2\}$. We
    regard~$\eta$ as tending to zero, so the conditions~\eqref{eq:phi_prime_at_zero}
    and~\eqref{eq:est_rho_0}--\eqref{eq:est_psi_0_0} imply that
    \begin{gather}\label{eq:est_rho_lu_0}
        \frac{1}{\rho_0(\eta)}=O\left(\sqrt\eta\right),\quad
        \rho_0(\eta) = O\left(\eta^{-\frac34}\right),
        \quad\text{and thus}\\[2pt]
        \label{eq:est_phi_0_0}
        \frac{\varphi(z_0) - \varphi(\overline z_0)}{(z_0 - \overline z_0)\rho_0^2(\eta)}
        =-1 + O\left(\eta\right).
    \end{gather}

    Now we make use of the same notation as in the proof of Proposition~\ref{prop:1}. If
    $m,n\ne 0$ and $m\ne n$, then the estimates~\eqref{eq:phi_n_n} and~\eqref{eq:phi_n_m_est}
    concerning $\varphi(z_1)$, \dots, $\varphi(z_\varkappa)$ are valid. Since the distance
    between~$U_n$ and~$z_0$ is more than~$\delta$, the inequality~\eqref{eq:bound_mixed} is
    satisfied on condition that~$m=0\ne n$. Then~\eqref{eq:bound_mixed}
    and~\eqref{eq:est_psi_0_n} give us the following:
    \begin{equation}\label{eq:est_phi_0_n}
    \begin{aligned}
            \left|
                \frac{\varphi(z_0) - \varphi(\overline z_n)}
                {z_0-\overline z_n}
            \right|
            \le{}&
            \left|
                \frac{\varphi_n(z_0) - \varphi_n(\overline z_n)}
                {z_0-\overline z_n}
            \right|
            +
            \left|
                \frac{\psi_0(z_0) - \psi_0(\overline z_n)}
                {z_0-\overline z_n}
            \right|\\
            &+
            \left|
                \frac{\big(\varphi(z)-\varphi_n(z_0)-\psi_0(z_0)\big)
                    - \big(\varphi(\overline z_n)-\varphi_n(\overline z_n)-\psi_0(\overline z_n)\big)}
                {z_0-\overline z_n}
            \right|
        \\
        \overset{\eqref{eq:bound_mixed}}\le{}&
        C(n,0,\delta)\,\rho_n(\eta)+
        \left|\frac{\psi_0(z_0) - \psi_0(\overline z_n)}
            {z_0 - \overline z_n}\right|
        +O\left(1\right)\\
        \xlongequal{\eqref{eq:est_psi_0_n}}{}&
        C(n,0,\delta)\,\rho_n(\eta)
        +O\Big(\eta^{-\frac 12}\Big)
        +O\left(\sqrt{\eta}\rho_0^2(\eta)\right)
        +O\left(1\right)
    \end{aligned}
    \end{equation}
    as~$\eta\to 0+$. Assume that~$\eta$ is taken from the
    sequence~$\big\{\eta_k\big\}_{k=1}^\infty$ corresponding to~\eqref{eq:lim_rho} and that the
    choice of~$\mu\in(0,\sqrt{\eta})$ satisfies the condition~\eqref{eq:choice_of_mu}.
    Then
    \[
    \begin{aligned}
        \left|\frac{\varphi(z_n) - \varphi(\overline z_0)}
            {(z_n-\overline z_0)\rho_n(\eta)\rho_0(\eta)}\right|
        ={}&
        \left|
            \frac{\varphi(z_0) - \varphi(\overline z_n)}
            {(z_0-\overline z_n)\rho_0(\eta)\rho_n(\eta)}
        \right|
        \\
        \xlongequal{\eqref{eq:est_phi_0_n}\text{ and }\eqref{eq:est_rho_lu_0}}{}&
        O\left(\sqrt\eta\right)
        +\left(O\Big(\eta^{-\frac12+\frac12}\Big)
        +O\Big(\eta^{\frac12-\frac34}\Big)
        +O\Big(\eta^{\frac12}\Big)\right)\frac 1{\rho_n(\eta)}\\
        \xlongequal{\eqref{eq:lim_rho}}{}&
        O\left(\sqrt\eta\right)
        +O\Big(\eta^{-\frac14}\Big)O\left(\sqrt\eta\right)
        =O\left(\sqrt[4]{\eta}\right).
    \end{aligned}
    \]
    This estimate together with~\eqref{eq:est_phi_0_0},~\eqref{eq:lim_rho},~\eqref{eq:phi_n_n}
    and~\eqref{eq:phi_n_m_est} yields that
    \[
    \begin{aligned}
    h\left(\frac{\zeta_0}{\rho_0(\eta)},\dots,\frac{\zeta_\varkappa}{\rho_\varkappa(\eta)}
        \Big\vert z_0,\dots,z_\varkappa\right)
        &= \!\sum_{n,m=1}^{\varkappa}
        \frac{\varphi(z_n) - \varphi(\overline z_m)}
             {(z_n-\overline z_m)\rho_n(\eta)\rho_m(\eta)}
             \,\zeta_n\overline\zeta_m\\
        &= - \sum_{m=0}^{\varkappa}|\zeta_m|^2 +O\left(\sqrt[4]{\eta}\right)
        \sum_{n,m=0}^{\varkappa} \zeta_m\overline\zeta_n,
    \end{aligned}
    \]
    where $O\left(\sqrt[4]{\eta}\right)$ does not depend on~$\zeta_0$, \dots, $\zeta_\varkappa$.
    That is, this Hermitian form is negative definite provided that the value
    of~$\eta\in\big\{\eta_k\big\}_{k=1}^\infty$ is small enough.
\end{proof}

\begin{proof}[Proof of Theorem~\ref{th:1}]
    Suppose that~$\Phi(z)=z\varphi(z)$ can be represented as in~\eqref{eq:phi_form_2_with_cond}.
    Then for specially chosen numbers~$z_1$, \dots, $z_\varkappa\notin\mathbb R$ the Hermitian
    form $h_\varphi(\xi_1, \dots, \xi_\varkappa \vert z_1,\dots, z_\varkappa)$ has $\varkappa$
    negative squares by Proposition~\ref{prop:1}. Let us show that this is the greatest possible
    number of negative squares in the form
    $h[\varphi]\colonequals h_{\varphi}(\xi_1, \dots, \xi_k \vert z_1,\dots, z_k)$.

    Denote $\widetilde\sigma(t)=\int_{0}^ts^{-1}d\sigma(s)$. Since the integral
    \[
    0 \le \int_0^\infty\frac{d\sigma(t)}{t(1+t^2)}
    =\int_0^\infty\frac{d\widetilde\sigma(t)}{1+t^2}
    \overset{\eqref{eq:phi_form_2_cond}}{{}\le{}}
      - a - \sum_{i=1}^{\varkappa}\frac{\sigma_i}{\lambda_i}<\infty
    \]
    is finite, we can split the last term of~\eqref{eq:phi_form_2} divided by~$z$ into two parts
    to obtain (cf.~\eqref{eq:ident_phi_int})
    \[
        \varphi(z)
        = b + \frac az + \sum_{i=1}^{\varkappa}
        \left(
            \frac{\sigma_i/\lambda_i}{\lambda_i-z}+\frac{\sigma_i/\lambda_i}{z}
        \right)
        + \int_0^\infty\frac{d\sigma(t)}{t(t-z)}
        + \frac 1z \int_0^\infty\frac{d\sigma(t)}{t(1+t^2)},
    \]
    that is~$\varphi(z)=\widetilde\varphi(z)+\varphi_0(z)$, where
    \[
        \widetilde\varphi(z)\colonequals
        \sum_{i=1}^{\varkappa}\frac{\sigma_i/\lambda_i}{\lambda_i-z}
        \quad\text{and}\quad
        \varphi_0(z) \colonequals
        b
        + \frac {\Phi(0-)}z
        + \int_0^\infty\frac{d\widetilde\sigma(t)}{t-z}.
    \]
    The functions $\varphi_0(z)$ and $-\widetilde\varphi(z)$ have the form~\eqref{eq:Phi_N},
    i.e. belong to the class~$\mathcal N$. For $\mathcal N$-functions and any set of
    numbers~$\{z_1, \dots, z_k\}$ the Hermitian form~\eqref{eq:quad_form} is nonnegative
    definite. That is, the conditions
    \begin{equation*}
        h[\varphi_0]\colonequals
        h_{\varphi_0}(\xi_1, \dots, \xi_k \vert z_1,\dots, z_k)\ge 0
            \quad\text{and}\quad
        h[\widetilde\varphi]\colonequals
        h_{\widetilde\varphi}(\xi_1, \dots, \xi_k \vert z_1,\dots, z_k)\le 0
    \end{equation*}
    holds true. Moreover, since~$\widetilde\varphi(z)$ is a rational function with $\varkappa$
    poles, which is bounded at infinity, the rank of~$h[\widetilde\varphi]$ can be at
    most~$\varkappa$ (see Theorem~3.3.3 and its proof in~\cite[pp.~105--108]{Akhiezer65} or
    Theorem~1 in~\cite[p.~34]{Donoghue}). Therefore, the
    form~$h[\varphi]= h_{\varphi}(\xi_1, \dots, \xi_k \vert z_1,\dots, z_k)$ has at
    most~$\varkappa$ negative squares as a sum of~$h[\varphi_0]$ and~$h[\widetilde\varphi]$.
    (This become evident after the reduction of the Hermitian form~$h[\widetilde\varphi]$ to
    principal axes since~$h[\varphi_0]$ is nonnegative definite irrespectively of coordinates.)

    Suppose that $\Phi(z)$ can be expressed as in~\eqref{eq:phi_form_1_with_cond}, and let
    $\varepsilon_0>0$ be such that~$\Phi(-\varepsilon)>0$ provided that
    $0<\varepsilon<\varepsilon_0$. In particular, it implies $\max_i\lambda_i<-\varepsilon_0$
    since $\Phi(\max_i\lambda_i+)<0$. Proposition~\ref{prop:2} provides a set of
    points~$\{z_0,\dots,z_{\varkappa-1}\}$ such that the corresponding Hermitian
    form~$h[\varphi]$ has~$\varkappa$ negative squares. Let us prove that~$h[\varphi]$ has at
    most~$\varkappa$ squares negative. Consider the function
    \[
    \varphi_\varepsilon(z)
    \colonequals 
    \frac{\Phi(z)-\Phi(-\varepsilon)}{z+\varepsilon} + \frac{\Phi(-\varepsilon)}{z+\varepsilon}
    \xlongequal{\eqref{eq:phi_form_1}}
    b + \frac{\Phi(-\varepsilon)}{z+\varepsilon}
    +\sum_{i=1}^{\varkappa-1}\frac{\sigma_i / (\lambda_i+\varepsilon)}{\lambda_i-z}
    + \int_0^\infty\frac{1}{t-z}\cdot\frac{d\sigma(t)}{t+\varepsilon}.
    \]
    Denote~$\Phi_\varepsilon(z)\colonequals z\varphi_\varepsilon(z)$ and
    $A_i\colonequals\dfrac{\sigma_i\lambda_i}{\lambda_i+\varepsilon}$, then
    \begin{align*}
        \Phi_\varepsilon(z)
        ={}& bz
        + \Phi(-\varepsilon) - \frac{\varepsilon\Phi(-\varepsilon)}{z+\varepsilon}
        + \sum_{i=1}^{\varkappa-1}\frac{A_iz / \lambda_i}{\lambda_i-z}
        + \int_0^\infty\frac{z}{t-z}\cdot\frac{d\sigma(t)}{t+\varepsilon}\\
        \xlongequal{\eqref{eq:ident_1tz}}{}& bz
        + \Phi(-\varepsilon) - \frac{\varepsilon\Phi(-\varepsilon)}{z+\varepsilon}
        + \sum_{i=1}^{\varkappa-1}\frac{A_i}{\lambda_i-z} - \sum_{i=1}^{\varkappa-1}\frac{A_i}{\lambda_i}
        + \int_0^\infty\left(\frac{1}{t-z}-\frac 1t\right)\frac{t\,d\sigma(t)}{t+\varepsilon}\\
        \xlongequal{\eqref{eq:ident_1t_2}}{}& bz
        + \left[
            \Phi(-\varepsilon) - \sum_{i=1}^{\varkappa-1}\frac{A_i}{\lambda_i}
            - \int_0^\infty\frac {d\sigma(t)}{(t+\varepsilon)(1+t^2)}
        \right]\\
        &- \frac{\varepsilon\Phi(-\varepsilon)}{z+\varepsilon}
        + \sum_{i=1}^{\varkappa-1}\frac{A_i}{\lambda_i-z}
        + \int_0^\infty\left(\frac{1}{t-z}-\frac t{1+t^2}\right)\frac{t\,d\sigma(t)}{t+\varepsilon},
    \end{align*}
    i.e. $\Phi_\varepsilon\in\mathcal N$. Moreover, $\Phi_\varepsilon(z)$ is an increasing
    function when $-\varepsilon<z<0$ (see Remark~\ref{rem:defN}) which implies
    \[
    \Phi_\varepsilon(0-)
    = \lim_{z\to0-}\int_0^\infty\frac{z}{t-z}\cdot\frac{d\sigma(t)}{t+\varepsilon}\le 0,
    \]
    since the integrand is negative. That is, the function $\Phi_\varepsilon(z)$ has the
    form~\eqref{eq:phi_form_2_with_cond}. As it is shown above, we have
    $\varphi_\varepsilon\in\mathcal N_\varkappa^+$ for each~$\varepsilon$ between~$0$
    and~$\varepsilon_0$.

    Given a fixed set of points~$\{z_1, \dots, z_k\}$ there exists some positive
    number~$\varepsilon_1<\varepsilon_0$, such that for all $0<\varepsilon<\varepsilon_1$ the
    form~$h[\varphi_\varepsilon]\colonequals h_{\varphi_\varepsilon}(\xi_1, \dots, \xi_k \vert
    z_1,\dots, z_k)$ has at least the same number of negative squares as the form~$h[\varphi]$.
    (Indeed: the characteristic numbers of~$h[\varphi]$ depend continuously on its
    coefficients.) Suppose that the Hermitian form~$h[\varphi]$ has more than~$\varkappa$
    negative squares. Then~$h[\varphi_\varepsilon]$ must have more than~$\varkappa$ negative
    squares as well, which is impossible. Thus, the form~$h[\varphi]$ has at most~$\varkappa$
    negative squares.

    Suppose that~$\varphi\in\mathcal N_\varkappa^+$. Then the function~$\Phi(z)=z\varphi(z)$ can
    be represented as in~\eqref{eq:Phi_N} and the form~$h[\varphi]$ for any set of
    numbers~$\{z_1, \dots, z_k\}$ has at most~$\varkappa$ negative squares (as stated in the
    definition of~$\mathcal N_\varkappa^+$). By Proposition~\ref{prop:1}, the
    function~$\sigma(t)$ appearing in~\eqref{eq:Phi_N} can have at most~$\varkappa$ negative
    points of increase. These points are isolated, and therefore (see
    Definition~\ref{def:poi_inc}) for negative~$t$ the function~$\sigma(t)$ is a step function
    with at most~$\varkappa$ steps. That is, all negative singular points of~$\Phi(z)$ are
    simple poles; they have negative residues since~$\Phi\in\mathcal N$, i.e. $\sigma_i>0$ for
    all~$i$. Here we have two mutually exclusive options:
    $\Phi(0-)\le0$, then $\Phi(z)$ has the form~\eqref{eq:phi_form_2_with_cond} corresponding to
    some~$\varkappa_0\le\varkappa$, and
    $0<\Phi(0-)\le\infty$, i.e. $\Phi(z)$ has the form~\eqref{eq:phi_form_1_with_cond}
    corresponding to~$\varkappa_0\le\varkappa+1$.
    The sufficiency (first) part of the current proof shows
    that~$\varphi\in\mathcal N_{\varkappa_0}^+$ in both cases. Since the
    classes~$\mathcal N_{\varkappa_0}^+$ and~$\mathcal N_{\varkappa}^+$ are disjoint by
    definition, we necessarily have~$\varkappa_0=\varkappa$.
\end{proof}

\section*{Acknowledgments}
This work appeared, inter alia, by virtue of my collaboration visits to Shanghai in 2014
(although they were devoted to other mathematical problems). I am grateful to Mikhail Tyaglov
for organizing these visits, to {\otherlanguage{ngerman}Technische Universit\"at Berlin} and
Shanghai Jiao Tong University for the financial support. I also thank Olga Holtz for her
encouragement and Victor Katsnelson for his attention.


\begin{thebibliography}{9}
\bibitem{Akhiezer65}
    N.\,I.~Akhiezer, \emph{The classical moment problem and some related questions in analysis}.
    (Oliver \& Boyd, Edinburgh-London, 1965).
        
\bibitem{Donoghue}W.\,F.~Donoghue, \emph{Monotone matrix functions and analytic continuation},
    (Springer-Verlag, Berlin-Heidelberg-New York, 1974).
    
\bibitem{Kac} Kac, I. S.,
    \emph{On integral representations of analytic functions mapping the upper half-plane onto a
        part of itself},
    Uspehi Mat. Nauk (N.S.) \textbf{11} (1956), No.~3(69), pp.~139--144 (Russian).
    \url{http://mi.mathnet.ru/eng/umn7798}

\bibitem{KreinLanger1}
    M.\,G.~Kre{\u\i}n and H.~Langer, {\otherlanguage{ngerman}
    \emph{\"Uber einige Fortsetzungsprobleme, die
        eng mit der Theorie hermitescher Operatoren im Raume $\Pi_\varkappa$ zusammenh\"angen. I.
        Einige Funktionenklassen und ihre Darstellungen}}, Math. Nachr. \textbf{77} (1977),
    pp.~187--236 (German).
    \url{http://dx.doi.org/10.1002/mana.19770770116}

    
\bibitem{LangerWinkler} H.~Langer and H.~Winkler, \emph{Direct and inverse spectral problems for
    generalized strings}, Integr. Equ. Oper. Theory {\bfseries 30} (1998) No.~4, pp.~409--431.
    \url{http://dx.doi.org/10.1007/BF01257875}
    
\end{thebibliography}
\end{document}